\documentclass[11pt]{article}
\usepackage[a4paper,left=2.5cm,right=2.5cm,top=2.5cm,bottom=2.5cm]{geometry}
\usepackage{amsmath,amssymb}
\usepackage{hyperref}
\usepackage{setspace}
\usepackage[utf8]{inputenc}
\usepackage[english]{babel}
\usepackage{booktabs,multirow}
\usepackage{blkarray}
\usepackage{lscape}
\usepackage[T1]{fontenc}
\usepackage{authblk,color}
\usepackage[english]{babel}
\usepackage{amsfonts, amstext, amsthm, booktabs, dcolumn}
\usepackage{graphicx} 
\usepackage{float}
\usepackage{caption}
\usepackage{subcaption}
\usepackage[mathscr]{euscript}
\usepackage{pgfplots}
\pgfplotsset{compat=1.7}

\newtheorem{proposition}{\bf Proposition}[section]
\newtheorem{theorem}{\bf Theorem}[section]

\newtheorem{remark}{\bf Remark}[section]

\newtheorem{lemma}{\bf Lemma}[section]

\newcommand\floor[1]{\left\lfloor#1\right\rfloor}

\begin{document}
\title{Approximation Results for Sums of Independent Random Variables}
\author[ ]{Pratima Eknath Kadu}
\affil[ ]{\small Department of Maths \& Stats,}
\affil[ ]{\small K. J. Somaiya College of Arts and Commerce,}
\affil[ ]{\small Vidyavihar, Mumbai-400077.}
\affil[ ]{\small Email: pratima.kadu@somaiya.edu}
\date{}
\maketitle

\begin{abstract}
\noindent
In this article, we consider Poisson and Poisson convoluted geometric approximation to the sums of $n$ independent random variables under moment conditions. We use Stein's method to derive the approximation results in total variation distance. The error bounds obtained are either comparable to or improvement over the existing bounds available in the literature. Also, we give an application to the waiting time distribution of 2-runs.
\end{abstract}

\noindent
\begin{keywords}
Poisson and geometric distribution; perturbations; probability generating function; Stein operator; Stein's method.
\end{keywords}\\
{\bf MSC 2010 Subject Classification:} Primary: 62E17, 62E20; Secondary: 60E05, 60F05.

\section{Introduction}
Let $\xi_1,\xi_2,\dotsc,\xi_n$ be $n$ independent random variables concentrated on $\mathbb{Z}_{+}=\{0,1,2,\ldots\}$ and 
\begin{equation}
W_n:=\sum_{i=1}^{n}\xi_i, \label{W}
\end{equation}
their convolution of $n$ independent random variables. The distribution of $W_n$ has received special attention in the literature due to its applicability in many settings such as rare events, the waiting time distributions, wireless communications, counts in nuclear decay, and business situations, among many others. For large values of $n$, it is in practice hard to obtain the exact distribution of $W_n$ in general, in fact, it becomes intractable if the underlying distribution is complicated such as hyper-geometric and logarithmic series distribution, among many others. It is therefore of interest to approximate the distribution of such $W_n$ with some well-known and easy to use distributions. Approximations to $W_n$ have been studied by several authors such as, saddle point approximation (Lugannani and Rice (1980) and Murakami (2015)), compound Poisson approximation (Barbour {\em et al.} (1992a), Serfozo (1986), and Roos (2003)), Poisson approximation (Barbour {\em et al.} (1992b)), the centred Poisson approximation (\v{C}ekanavi\v{c}ius and Vaitkus (2001)), compound negative binomial approximation (Vellaisamy and Upadhye (2009)), and negative binomial approximation (Vellaisamy {\em et al.} (2013) and Kumar and Upadhye (2017)).\\
In this article, we consider Poisson and Poisson convoluted geometric approximation to $W_n$. Let $X$ and $Y$ follow Poisson and geometric distribution with parameter $\lambda$ and $p=1-q$ with probability mass function (PMF)
\begin{align}
P(X=k)=\frac{e^{-\lambda} \lambda^k}{k!}\quad\text{and}\quad P(Y=k)=q^k p, \quad k=0,1,2,\dotsc,\label{ccd}
\end{align}
respectively.  Also, assume $X$ and $Y$ are independent. We use Stein's method to obtain bounds for the approximation of the law of $W_n$ with that of $X$ and $X+Y$. Stein's method (Stein (1986)) requires identification of a Stein operator and there are several approaches to obtain Stein operators (see Reinert (2005)) such as density approach (Stein (1986), Stein {\em et al.} (2004), Ley and Swan (2013a, 2013b)), generator approach (Barbour (1990) and G\"{o}tze (1991)), orthogonal polynomial approach (Diaconis and Zabell (1991)), and probability generating function (PGF) approach (Upadhye {\em et al.} (2014)). We use the PGF approach to obtain Stein operators.\\
This article is organized as follows. In Section \ref{sectwo}, we introduce some notations to simplify the presentation of the article. Also, we discuss some known results of Stein's method. In Section \ref{secthree}, Stein operators for $W_n$ and $X+Y$ are obtained as a perturbation of the Poisson operator. In Section \ref{secfour}, the error bounds for $X$ and $X+Y$ approximation to $W_n$ are derived in total variation distance. In Section \ref{tsec}, we demonstrate the relevance of our results through an application to the waiting time distribution of 2-runs. In Section \ref{psec}, we point out some relevant remarks.

\section{Notations and Preliminaries}\label{sectwo}
Recall that $W_n=\sum_{i=1}^{n}\xi_{i}$, where $\xi_{1},\xi_{2},\dots,\xi_{n}$ are $n$ independent random variables concentrated on $\mathbb{Z}_+$. Throughout, we assume that $\psi_{\xi_i}$, the PGF of $\xi_i$, satisfies
\begin{equation}
\frac{\psi'_{\xi_i}(w)}{\psi_{\xi_i}(w)}=\sum_{j=0}^{\infty}g_{i,j+1}w^j=:\phi_{\xi_i}(w), \label{pgf-xi}
\end{equation}
at all $w \in \mathbb{Z}_{+}$. Note that this assumption is satisfied for the series \eqref{pgf-xi} converges absolutely. Also, one can show that the hyper-geometric and logarithmic series distribution do not satisfy \eqref{pgf-xi}. See Yakshyavichus (1998), and Kumar and Upadhye (2017) for more details. Note that
\begin{enumerate}
\item If $\xi_i \sim Po(\lambda_i) \implies g_{i,j+1} = \left\{
\begin{array}{ll}
\lambda_i,  & \text{for $j=0$},\\
0, & \text{for $j\ge 1$}.
\end{array}\right.$
\item If $\xi_i \sim Ge(p_i) \implies g_{i,j+1} = q_i^{j+1}$.
\item If $\xi_i \sim Bi(n,p_i) \implies g_{i,j+1} = n (-1)^j \left(p_i/(1-p_i)\right)^{j+1}$.
\end{enumerate}
Next, let $\mu$ and ${\sigma}^{2}$ be the mean and variance of $W_n$, respectively. Also, let $\mu_2$ and $\mu_3$ denote the second and third factorial cumulant moments of $W_n$, respectively. Then, it can be easily verified that
\begin{align}
\mu&=\sum_{i=1}^{n} \phi_{\xi_i}(1)=\sum_{i=1}^{n}\sum_{j=0}^{\infty} g_{i,j+1}, \
{\sigma}^2=\sum_{i=1}^{n}\Big[ \phi_{\xi_i}(1)+\phi'_{\xi_i}(1) \Big]=\sum_{i=1}^{n}\sum_{j=0}^{\infty}(j+1)g_{i,j+1},\label{muu} \\
\mu_2&=\sum_{i=1}^{n} \phi'_{\xi_i}(1)=\sum_{i=1}^{n}\sum_{j=0}^{\infty} jg_{i,j+1}, \
{\rm and}\
\mu_3=\sum_{i=1}^{n} \phi''_{\xi_i}(1)=\sum_{i=1}^{n}\sum_{j=0}^{\infty} j(j-1)g_{i,j+1}.\nonumber
\end{align}
For more details, see Vellaisamy {\em et al.} (2013), and Kumar and Upadhye (2017).\\
Next, let $H:=\{f|f:\mathbb{Z}_{+}\rightarrow \mathbb{R} \text{ is bounded}\}$ and
\begin{equation}
H_{\bar X}:=\{h \in H|h(0)=0, ~\text{and}~ h(j)=0 ~\text{for}~ j \notin Supp({\bar X})\}\label{Hw}
\end{equation}
for a random variable ${\bar X}$ and $Supp({\bar X})$ denotes the support of random variable ${\bar X}$.\\ 
Now, we discuss Stein's method which can be carried out in the following three steps. \\
We  first identify a suitable operator $\mathscr{A}_{\bar X}$ for a random variable $\bar{X}$ (known as Stein operator) such that 
$$\mathbb{E}(\mathscr{A}_{\bar X}h({\bar X}))=0, \quad \text{for $h \in H$.}$$
In the second step, we find a solution to the Stein equation
\begin{equation}
\mathscr{A}_{\bar X}h(j)=f(j)-\mathbb{E}f({\bar X}),~ j \in \mathbb{Z}_+ ~ {\rm and} ~ f \in H_{\bar{X}} \label{st eq}
\end{equation}
and obtain the bound for $\|\Delta h\|$, where $\|\Delta h\|= \sup_{j \in \mathbb{Z}_+}|\Delta h(j)|$ and $\Delta h(j)=h(j+1)-h(j)$ denotes the first forward difference operator.\\
Finally, substitute a random variable ${\bar Y}$ for $j$ in \eqref{st eq} and taking expectation and supremum, the expression leads to
\begin{equation}
d_{TV}({\bar X},{\bar Y}):=\sup_{f \in \mathcal{H}} \big| \mathbb{E}f({\bar X})- \mathbb{E}f({\bar Y}) \big| =\sup_{f \in \mathcal{H}} \big| \mathbb{E}[\mathscr{A}_{\bar X}h({\bar Y})] \big|, \label{dist-xy}
\end{equation}
where $\mathcal{H}=\{{\bf 1}_A~|~A\subseteq \mathbb{Z}_+\}$ and ${\bf 1}_A$ is the indicator function of $A$. Equivalently, \eqref{dist-xy} can be represented as
\begin{align*}
d_{TV}({\bar X},{\bar Y})=\frac{1}{2}\sum_{j=0}^{\infty} |P({\bar X}=j)-P({\bar Y}=j)|.
\end{align*}
For more details, we refer the reader to Barbour {\em et. al.} (1992b), Chen {\em et. al.} (2011), Goldstein and Reinert (2005), and Ross (2011). For recent developments, see Barbour and Chen (2014),  Ley {\em et. al.} (2014), Upadhye {\em et. al.} (2014), and references therein.\\
Next, it is known that a Stein operator for $X \sim Po(\lambda)$, the Poisson random variable with parameter $\lambda$, is given by
\begin{equation}
\mathscr{A}_Xh(j)=\lambda h(j+1)-jh(j), \quad {\rm for}~ j \in \mathbb{Z}_+ ~{\rm and}~ h \in H. \label{st-Pox}
\end{equation}
Also, from Section $5$ of Barbour and Eagleson (1983), the bound for the solution to the stein equation (say $h_f$) is given by
\begin{equation}
\|\Delta h_f\| \leq  \frac{1}{\max(1,\lambda)},\quad \text{for $f\in \mathcal{H}$, $h\in H$}. \label{bound-Pox}
\end{equation}
In terms of $\|f\|$, we have the following bound
\begin{equation}
\|\Delta h_f\| \leq  \frac{2\|f\|}{\max(1,\lambda)},\quad \text{for $f\in \mathcal{H}$, $h\in H$}. \label{bound-Pox1}
\end{equation}
See Section 3 of Upadhye {\em et al.} (2014) for more details. Note that the condition $h(0)=0$ in \eqref{Hw} is used while obtaining the bound \eqref{bound-Pox}, see Barbour and Eagleson (1983) for more details.\\
Next, suppose we have three random variables $X_1$, $X_2$, and $X_3$ defined on some common probability space. Define $\mathscr{U}=\mathscr{A}_{X_2}-\mathscr{A}_{X_1}$ then the upper bound for $d_{TV}(X_2,X_3)$ can be obtained by the following lemma  which is given by Upadhye {\em et al.} (2014).

\begin{lemma}\label{perturbation} [Lemma $3.1$, Upadhye et al. (2014)]
Let $X_1$ be a random variable with support ${\cal S}$, Stein operator $ \mathscr{A}_{X_1}$, and $h_0$ be the solution to Stein equation \eqref{st eq} satisfying 
$$\|\Delta h_0\|\le w_1 \|f\| \min(1,\alpha^{-1}),$$
where $w_1,~\alpha>0$. Also, let $X_2$ be a random variable whose Stein operator can be written as $ \mathscr{A}_{X_2}=\mathscr{A}_{X_1}+U_1$ and $X_3$ be a random variable such that, for $h\in H_{X_1}\cap H_{X_2}$, 
$$\|U_1h\|\le w_2\|\Delta h\|\quad \text{and}\quad|{\mathbb E}\mathscr{A}_{X_2}h(X_3)|\le \varepsilon \|\Delta h\|,$$
where $w_1w_2<\alpha$. Then
$$d_{TV}(X_2,X_3)\le \frac{\alpha}{2(\alpha-w_1 w_2)}\left(\varepsilon w_1 \min(1,\alpha^{-1})+2 P(X_2\in {\cal S}^c)+2 P(X_3\in {\cal S}^c)\right),$$
where ${\cal S}^c$ denote the complement of set ${\cal S}$.
\end{lemma}
\noindent
Finally, from Corollary 1.6 of Mattner and Roos (2007), we have
\begin{equation}
d_{TV}(W_n,W_n+1) \leq \sqrt{\frac{2}{\pi}} \Bigg(\frac{1}{4}+\sum_{i=1}^{n} \big(1-d_{TV}(\xi_{i},\xi_{i}+1)\big)\Bigg)^{-1/2}.\label{dist-w,w+1}
\end{equation}
For more details about these results, we refer the reader to Barbour {\em et al.} (2007), Upadhye {\em et al.} (2014), Vellaisamy {\em et al.} (2013), Kumar and Upadhye (2017), and references therein.

\section{Stein Operator for the Convolution of Random Variables} \label{secthree}
In this section, we derive Stein operators for $W_n$ and $X+Y$ as a perturbation of Poisson operator which are used to obtain the main results in Section \ref{secfour}.
\begin{proposition}\label{st-w1}
Let $\xi_1,\xi_2,\dotsc,\xi_n$ be independent random variables satisfying \eqref{pgf-xi} and $W_n=\sum_{i=1}^{n}\xi_i$. Then, a Stein operator for $W_n$ is
\begin{align*}
\mathscr{A}_{W_n}h(j)={\mu}h(j+1)-jh(j)+\sum_{i=1}^{n}\sum_{k=0}^{\infty}\sum_{l=1}^{k}g_{i,k+1} \Delta h(j+l),
\end{align*}
where $\mu$ is defined in \eqref{muu}.
\end{proposition}
\begin{proof}
It can be easily verified that the PGF of $W_n$, denoted by $\psi_{W_n}$, is 
$$\psi_{W_n}(w)=\prod_{i=1}^{n}\psi_{\xi_i}(w)$$
as $\xi_1,\xi_2,\dotsc,\xi_n$ are independent random variables. Differentiating with respect to $w$, we have
\begin{align*}
\psi^{'}_{W_n}(w)&=\psi_{W_n}(w)\sum_{i=1}^{n} \phi_{\xi_i}(w)\\
&=\sum_{i=1}^{n}\psi_{W_n}(w)
\sum_{j=0}^{\infty}g_{i,j+1}w^j,
\end{align*}
where $\phi_{\xi_i}(\cdot)$ is defined in \eqref{pgf-xi}.
Using definition of the PGF, the above expression can be expressed as
\begin{align*}
\sum_{j=0}^{\infty}(j+1)\gamma_{j+1}w^{j}=\sum_{i=1}^{n}\sum_{k=0}^{\infty}\gamma_{k}w^{k}
\sum_{j=0}^{\infty}g_{i,j+1}w^{j}=\sum_{j=0}^{\infty}\left(\sum_{i=1}^{n}\sum_{k=0}^{j}\gamma_{k}g_{i,j-k+1}\right)w^{j},
\end{align*}
\noindent
where $\gamma_j=P(W_n=j)$.
Comparing the coefficients of $w^{j}$, we get
$$\sum_{i=1}^{n}\sum_{k=0}^{j}\gamma_{k}g_{i,j-k+1}-(j+1)\gamma_{j+1}=0.$$
Let $h \in H_{W_n}$ as defined in \eqref{Hw}, then
$$\sum_{j=0}^{\infty}h(j+1)\left[\sum_{i=1}^{n}\sum_{k=0}^{j}\gamma_{k}g_{i,j-k+1}-(j+1)\gamma_{j+1}\right]=0.$$
Therefore,
$$\sum_{j=0}^{\infty}\left[\sum_{i=1}^{n}\sum_{k=0}^{\infty}g_{i,k+1}h(j+k+1)-jh(j)\right] \gamma_j =0.$$
Hence, a Stein operator for $W_n$ is given by
\begin{equation}
\mathscr{A}_{W_n}h(j)=\sum_{i=1}^{n}\sum_{k=0}^{\infty}g_{i,k+1}h(j+k+1)-jh(j). \label{st-w}
\end{equation}
It is well known that 
\begin{equation}
h(j+k+1)=\sum_{l=1}^{k}\Delta h(j+l)+h(j+1). \label{grad}
\end{equation}
Using \eqref{grad} in \eqref{st-w}, the proof follows.
\end{proof}

\begin{proposition} \label{st-z}
Let $X\sim \text{Po}(\lambda)$ and $Y\sim \text{Ge}(p)$ as defined in \eqref{ccd}. Also, assume $X$ and $Y$ are independent random variables. Then a Stein operator for $X+Y$ is given by
\begin{align*}
\bar{\mathscr{A}}_{X+Y}h(j)=\Big(\lambda+\frac{q}{p}\Big)h(j+1)-jh(j)+\sum_{k=0}^{\infty}\sum_{l=1}^{k} q^{k+1} \Delta h(j+l).
\end{align*}
\end{proposition}
\begin{proof}
It is known that the PGF of $X$ and $Y$ are 
$$\psi_X(w)=e^{-\lambda(1-w)}\quad\text{and}\quad \psi_Y(w)=\frac{p}{1-qw},$$ respectively. Then, the PGF of $Z=X+Y$ is given by
$$\psi_{Z}(w)=\psi_X(w).\psi_Y(w).$$
Differentiating with respect to $w$, we get
\begin{align*}
\psi'_{Z}(w)=\Big(\lambda+\frac{q}{1-qw}\Big)\psi_{Z}(w)=\Big(\lambda+q\sum_{j=0}^{\infty}q^jw^j\Big)\psi_{Z}(w),\quad |w|<q^{-1}.
\end{align*}
Let $\bar{\gamma}_{j}=P(Z=j)$ be the PMF of $Z$. Then, using definition of the PGF, we have
$$\sum_{j=0}^{\infty}(j+1) \bar{\gamma}_{j+1}w^{j}=\lambda\sum_{j=0}^{\infty} \bar{\gamma}_{j}w^{j}+\sum_{j=0}^{\infty}q^{j+1}w^{j}\sum_{k=0}^{\infty} \bar{\gamma}_{k}w^{k}.$$
This implies
$$\sum_{j=0}^{\infty}(j+1) \bar{\gamma}_{j+1}w^{j}-\lambda\sum_{j=0}^{\infty} \bar{\gamma}_{j}w^{j}-\sum_{j=0}^{\infty}\left(\sum_{k=0}^{j} \bar{\gamma}_{k}q^{j-k+1}\right)w^{j}=0.$$
Collecting the coefficients of $w^j$, we get
$$(j+1) \bar{\gamma}_{j+1}-\lambda \bar{\gamma}_{j}-\sum_{k=0}^{j} \bar{\gamma}_{k}q^{j-k+1}=0.$$
Let $h \in H_{Z}$ as defined in \eqref{Hw}, then
$$\sum_{j=0}^{\infty}h(j+1)\Big[\lambda \bar{\gamma}_{j}-(j+1) \bar{\gamma}_{j+1}+\sum_{k=0}^{j} \bar{\gamma}_{k}q^{j-k+1}\Big]=0.$$
Further simplification leads to
$$\sum_{j=0}^{\infty}\Big[\lambda h(j+1)-jh(j)+\sum_{k=0}^{\infty}q^{k+1}h(j+k+1)\Big] \bar{\gamma}_{j}=0.$$
Therefore,
$$\bar{\mathscr{A}}_{X+Y}h(j)=\lambda h(j+1)-jh(j)+\sum_{k=0}^{\infty}q^{k+1}h(j+k+1).$$
Using \eqref{grad}, the proof follows.
\end{proof}

\section{Approximation Results} \label{secfour}
In this section, we derive an error bound for the Poisson and Poisson convoluted geometric approximation to $W_n$.  The following theorem gives the bound for Poisson, with parameter $\mu$, approximation.

\begin{theorem}\label{bound-po}
Let $\xi_1,\xi_2,\dotsc,\xi_n$ be independent random variables satisfying \eqref{pgf-xi} and $W_n=\sum_{i=1}^{n}\xi_i$. Then
\begin{align*}
d_{TV}(W_n,X) \le \frac{\left| \mu_2\right| }{\max (1,\mu)},
\end{align*}
where $X \sim Po(\mu)$.
\end{theorem}
\begin{proof}
From Proposition \ref{st-w1}, a Stein operator for $W_n$ is given by
\begin{align*}
\mathscr{A}_{W_n}h(j)&={\mu}h(j+1)-jh(j)+\sum_{i=1}^{n}\sum_{k=0}^{\infty}\sum_{l=1}^{k} g_{i,k+1} \Delta h(j+l)=\mathscr{A}_{X}h(j)+\mathscr{U}_{W_n}h(j),
\end{align*}
where $\mathscr{A}_X$ is a Stein operator for $X$ as discussed in \eqref{st-Pox}. Observe that $\mathscr{A}_{W_n}$ is a Stein operator for $W_n$ which can be seen as a perturbation of Poisson operator. Now, for $h \in H_X \cap H_{W_n}$, taking expectation of perturbed operator $\mathscr{U}_{W_n}$ with respect to $W_n$ and using \eqref{bound-Pox}, the result follows.
\end{proof}
\noindent
Next, we derive $Z=X+Y$ approximation to $W_n$, where $X\sim\text{Po}(\lambda)$ and $Y\sim\text{Ge}(p)$, by matching first two moments, that is, $ \mathbb{E}(Z)= \mathbb{E}(W_n)$ and $\mathrm{Var}(Z)=\mathrm{Var}(W_n)$ which give the following choice of parameters
\begin{align}
\lambda=\mu-\sqrt{\sigma^2-\mu}\quad \text{and}\quad p=\frac{1}{1+\sqrt{\sigma^2-\mu}}.\label{pp:1}
\end{align}
\begin{theorem} \label{bound-P1}
Let $\xi_1,\xi_2,\dotsc,\xi_n$ be independent random variables satisfying \eqref{pgf-xi} and the mean and variance of $W_n=\sum_{i=1}^{n}\xi_i$ satisfying  \eqref{pp:1}. Also, assume that $\sigma^2>\mu$ and $\lambda >2(q/p)^2$. Then
\begin{align*}
d_{TV}(W_n,Z) \leq  \frac{\lambda\sqrt{\frac{2}{\pi}} \Big| \mu_{3}-2 \left(q/p\right)^3 \Big|\left(\frac{1}{4}+\sum_{i=1}^{n} \big(1-d_{TV}(\xi_{i},\xi_{i}+1)\big)\right)^{-1/2}}{{ \Big( \lambda} -2(q/p)^2 \Big)\max(1,\lambda)},
\end{align*}
where $Z=X+Y$, $X\sim \text{Po}(\lambda)$ and $Y\sim \text{Ge}(p)$.
\end{theorem}

\begin{remark}
Note that, in Theorem \ref{bound-P1}, the choice of parameters are valid as 
$$\mu=\lambda+\frac{q}{p} >\frac{q}{p}=\sqrt{\sigma^2-\mu}\quad \text{and}\quad p=\frac{1}{1+\sqrt{\sigma^2-\mu}}\le 1,$$
since $\sigma^2>\mu$. 
\end{remark}

\begin{proof}[Proof of Theorem \ref{bound-P1}]
From \eqref{st-w}, the Stein operator for $W_n$ is given by
$$\mathscr{A}_{W_n}h(j)=\sum_{i=1}^{n}\sum_{k=0}^{\infty}g_{i,k+1}h(j+k+1)-jh(j).$$
Using \eqref{grad}, with $\sum_{i=1}^{n} \sum_{k=0}^{\infty} g_{i,k+1}= \mathbb{E}(W_n)= \mathbb{E}(Z)= \lambda + q/p$, we get
\begin{align*}
\mathscr{A}_{W_n}h(j)&=\Big(\lambda+\frac{q}{p}\Big)h(j+1)-jh(j)+\sum_{k=0}^{\infty}\sum_{l=1}^{k}q^{k+1}\Delta h(j+l)\\
&~~~+\sum_{i=1}^{n}\sum_{k=0}^{\infty}\sum_{l=1}^{k}g_{i,k+1}\Delta h(j+l)-\sum_{k=0}^{\infty}\sum_{l=1}^{k}q^{k+1}\Delta h(j+l)\\
&=\mathscr{A}_{Z}h(j)+\bar{\mathscr U}_{W_n}h(j).
\end{align*}
This is a Stein operator for $W_n$ which can be seen as perturbation of $Z=X+Y$ operator, obtained in Proposition \ref{st-z}.
Now, consider
\begin{equation}
\bar{\mathscr U}_{W_n}h(j)= \sum_{i=1}^{n}\sum_{k=0}^{\infty}\sum_{l=1}^{k}g_{i,k+1}\Delta h(j+l)-\sum_{k=0}^{\infty}\sum_{l=1}^{k}q^{k+1}\Delta h(j+l). \label{u}
\end{equation}
We know that
$$\Delta h(j+l)=\sum_{m=1}^{l-1}{\Delta}^{2} h(j+m)+\Delta h(j+1).$$
Substituting in \eqref{u} and using $\mathrm{Var}(Z)=\mathrm{Var}(W_n)$ with $\sum_{i=1}^{n} \sum_{k=0}^{\infty} g_{i,k+1}= \mathbb{E}(W_n)= \mathbb{E}(Z)= \lambda + q/p$, we have
\begin{align*}
\bar{\mathscr U}_{W_n}h(j)=\sum_{i=1}^{n}\sum_{k=0}^{\infty}\sum_{l=1}^{k}\sum_{m=1}^{l-1}g_{i,k+1}{\Delta}^{2} h(j+m)-\sum_{k=0}^{\infty}\sum_{l=1}^{k}\sum_{m=1}^{l-1}q^{k+1}{\Delta}^{2} h(j+m).
\end{align*}
Now, taking expectation with respect to $W_n$, we get
\begin{align*}
\mathbb{E}\big[\bar{\mathscr{U}}_{W_n}h(W_n)\big]=&\sum_{j=0}^{\infty}\Big[ \sum_{i=1}^{n}\sum_{k=0}^{\infty}\sum_{l=1}^{k}\sum_{m=1}^{l-1} g_{i,k+1}{\Delta}^2 h(j+m)\\
&-\sum_{k=0}^{\infty}\sum_{l=1}^{k}\sum_{m=1}^{l-1} q^{k+1}{\Delta}^2h(j+m)\Big]P[W_n=j].
\end{align*}
Therefore,
\begin{align*}
\left|\mathbb{E}\big[\bar{\mathscr{U}}_{W_n}h(W_n)\big]\right|&\leq \ 2d_{TV}(W_n,W_n+1)\|\Delta h\| \Bigg|\sum_{i=1}^{n}\sum_{k=0}^{\infty} \frac{k(k-1)}{2} g_{i,k+1}-\sum_{k=0}^{\infty} \frac{k(k-1)}{2}q^{k+1}\Bigg|.\\
&\leq \ d_{TV}(W_n,W_n+1)\|\Delta h\| \Bigg| \mu_{3}-2 \frac{q^3}{p^3} \Bigg|.
\end{align*}
Using \eqref{dist-w,w+1}, we have
\begin{align}
\Big|\mathbb{E}\big[\mathscr{U}_{W_n}h(W_n)\big]\Big| \leq \|\Delta h\|\sqrt{\frac{2}{\pi}} \Bigg(\frac{1}{4}+\sum_{i=1}^{n} \big(1-d_{TV}(\xi_{i},\xi_{i}+1)\big)\Bigg)^{-1/2} \Bigg| \mu_{3}-2 \frac{q^3}{p^3} \Bigg|.\label{delta2}
\end{align}
From Proposition \ref{st-z}, we have
\begin{align}
\|{\cal U}_{X+Y}h\|\le \frac{q^2}{p^2}\|\Delta h\|.\label{hhh}
\end{align}
Using \eqref{bound-Pox1}, \eqref{delta2}, and \eqref{hhh} with Lemma \ref{perturbation}, the proof follows.
\end{proof}

\section{An Application to the Waiting Time Distribution of 2-runs}\label{tsec}
The concept of runs and patterns is well-known in the literature due to its applicability in many real-life applications such as reliability theory, machine maintenance, statistical testing, and quality control, among many others. In this section, we consider the set up discussed by Hirano (1984) and generalized by Huang and Tsai (1991) as follows:\\
Let $N$ denote the number of two consecutive successes in $n$ Bernoulli trials with success probability $p$.  Then, Huang and Tsai (1991) (with $k_1=0$ and $k_2=2$ in their notation) have shown that the waiting time for $n$th occurrence of 2-runs can be written as the sum of $n$ independent and identical distributed (iid) random variables, say $U_1,U_2,\dotsc,U_n$, concentrated on $\{2,3,\dotsc\}$.  Here $U_i$ is 2 plus the number of trials between the $(j - 1)$th and $j$th occurrence of 2-runs. The PGF of $U_i$ is given by
\begin{align*}
\psi_U(t)=\frac{p^2t^2}{1-t+p^2t^2},
\end{align*} 
where $U$ is the iid copy of $U_i$, $i=1,2,\dotsc,n$ (see Hung and Tsai (1991) for more details). Now, let $V_i=U_i-2$ concentrated on $\mathbb{Z}_+$. Then, Kumar and Upadhye (2017) have given the PGF of $V_i$ and which is given by
\begin{align*}
\psi_{V_i}(t)=\frac{p^2}{1-t+p^2t^2}=\sum_{j=0}^{\infty}\left(\sum_{\ell=0}^{\floor{j/2}}\binom{j-\ell}{\ell}(-1)^\ell p^{2 (\ell+1)}\right)t^j=\sum_{j=0}^{\infty}g_{i,j+1}t^j,
\end{align*} 
where $g_{i,j+1}=\sum_{\ell=0}^{\floor{j/2}}\binom{j-\ell}{\ell}(-1)^\ell p^{2 (\ell+1)}$, for each $i=1,2,\dotsc,n$. For more details, we refer the reader to Huang and Tsai (1991), Kumar and Upadhye (2017), and Balakrishnan and Koutras (2002), and references therein.\\
Now, let $W_{\bar{n}}=\sum_{i=1}^{\bar{n}}V_i$ then $W_{\bar{n}}$ denotes the number of failures before ${\bar{n}}^\text{th}$ occurrence of 2-runs. Therefore, from Theorem \ref{bound-po}, we have
\begin{align*}
d_{TV}(W_{\bar{n}},Po(\mu))\le \frac{|\mu_2|}{\max(1,\mu)},
\end{align*}
where $\mu=\bar{n}\sum_{j=0}^{\infty}g_{i,j+1}$ and $\mu_2=\bar{n}\sum_{j=0}^{\infty}jg_{i,j+1}$. In a similar manner, from Theorem \ref{bound-P1}, we can also obtain the bound for the Poisson convoluted geometric approximation. For more details, we refer the reader to Section $4$ of Kumar and Upadhye (2017).

\section{Concluding Remarks}\label{psec}
\begin{enumerate}
\item Note that, if $\xi_{i} \sim Po(\lambda_{i})$, $i=1,2,\dotsc,n$ then $ d_{TV}(W_n,X)=0$ in Theorem \ref{bound-po}, as expected.
\item If $\xi_{1} \sim Po(\lambda)$ and $\xi_{2}\sim Ge(p)$, for $i=1,2$, and $W_2=\xi_1+\xi_2$ then $d_{TV}(W_2,Z)=0$ in Theorem \ref{bound-P1}, as expected.
\item The bounds obtained in Theorems \ref{bound-po} and \ref{bound-P1} are either comparable to or improvement over the existing bounds available in the literature. In particular, some comparison can be seen as follows:
\begin{enumerate}
\item If $\xi_i\sim Ber(p_i)$, for $i=1,2,\dotsc,n$ then, from Theorem \ref{bound-po}, we have
\begin{align*}
d_{TV}(W_n,Po(\mu))\le \frac{1}{\max(1,\mu)}\sum_{i=1}^{n}p_i^2,
\end{align*}
where $\mu=\sum_{i=1}^{n}p_i$. The above bound is same as given by Barbour {\em et al.} (1992b) and is an improvement over the bound $d_{TV}(W_n,Po(\mu))\le \sum_{i=1}^{n}p_i^2$ given by Khintchine (1933) and Le Cam (1960).
\item If $\xi_i\sim Ge(p_i)$, $i=1,2,\dotsc,n$ then, from Theorem \ref{bound-po}, we have
$$d_{TV}(W_n,X)\le \frac{1}{\max(1,\mu)}\sum_{i=1}^{n} \left( \frac{q_i}{p_i}\right)^2.$$
This bound is an improvement over negative binomial approximation given by Kumar and Upadhye (2017) in Corollary 3.1. 
\item If $\xi_i\sim NB(\alpha_i,p_i)$, $i=1,2,\dotsc,n$ then, from Theorems \ref{bound-po}, we have
\begin{align}
d_{TV}(W_n,Po(\mu))\le \frac{1}{\max(1,\mu)}\sum_{i=1}^{n}\alpha_i \left( \frac{q_i}{p_i}\right)^2,\label{vv}
\end{align}
where $\mu=\sum_{i=1}^{n}\frac{\alpha_i q_i}{p_i}$. Vellaisamy and Upadhye (2009) obtained bound for $S_n=\sum_{i=1}^n \xi_i$ and is given by
\begin{equation}
d_{TV}(S_n,Po(\lambda))\le { \min \left(1, \frac{1}{\sqrt{2 \lambda e}} \right)} \sum_{i=1}^{n} \frac{\alpha_i q_i^2}{p_i}, \label{v1}
\end{equation}
where $ \lambda = \sum_{i=1}^{n} \alpha_i q_i = \alpha q$. Under identical set up with $\alpha=5$ and various values of $n$ and $q$, the numerical comparison of \eqref{vv} and \eqref{v1} as follows: 

\begin{table}[H]
  \centering
  \caption{Comparison of bounds.}
\begin{tabular}{lccccccccccc}
 \toprule
$n$ & $q$ & From \eqref{vv} & From \eqref{v1} \\
\midrule
\vspace{0.05cm}
10 & \multirow{3}{*}{0.1} &0.1111& 0.3370 \\
\vspace{0.05cm}
30 &  & 0.1111 & 1.0109 \\
\vspace{0.05cm}
50 &  & 0.1111& 1.6848 \\
\midrule
\vspace{0.05cm}
10& \multirow{3}{*}{0.2} &0.2500 & 1.0722 \\
\vspace{0.05cm}
30&  &0.2500  & 3.2166 \\
\vspace{0.05cm}
50&  &  0.2500& 5.3610 \\
\bottomrule
\end{tabular}
\end{table}

\noindent
Note that our bound (from \eqref{vv}) is better than the bound given in \eqref{v1}. In particular, graphically, the closeness of these two distributions can be seen as follows:

\begin{figure}[H]
\centering
\begin{minipage}{.4\textwidth}
  \centering
  \includegraphics[width=0.9 \linewidth]{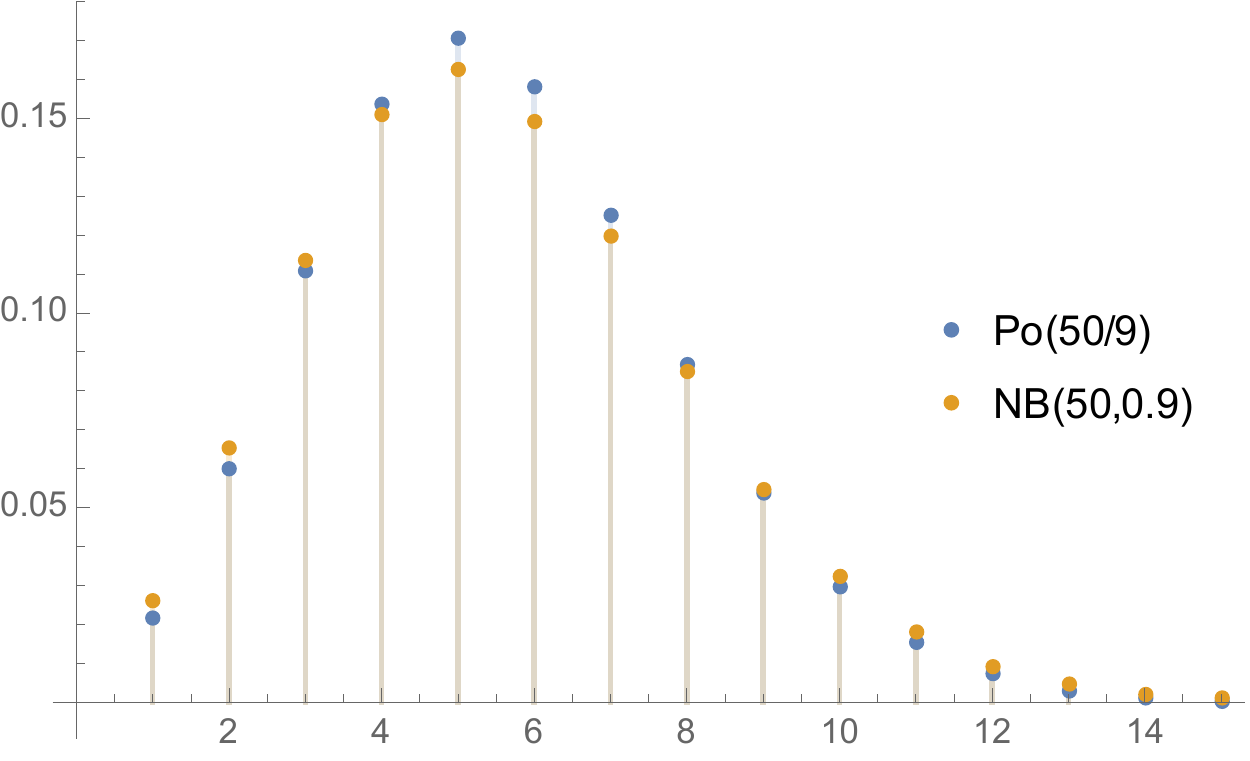}
  \captionof{figure}{$n=10, q=0.1$}
\end{minipage}%
\begin{minipage}{.4\textwidth}
  \centering
  \includegraphics[width=0.9 \linewidth]{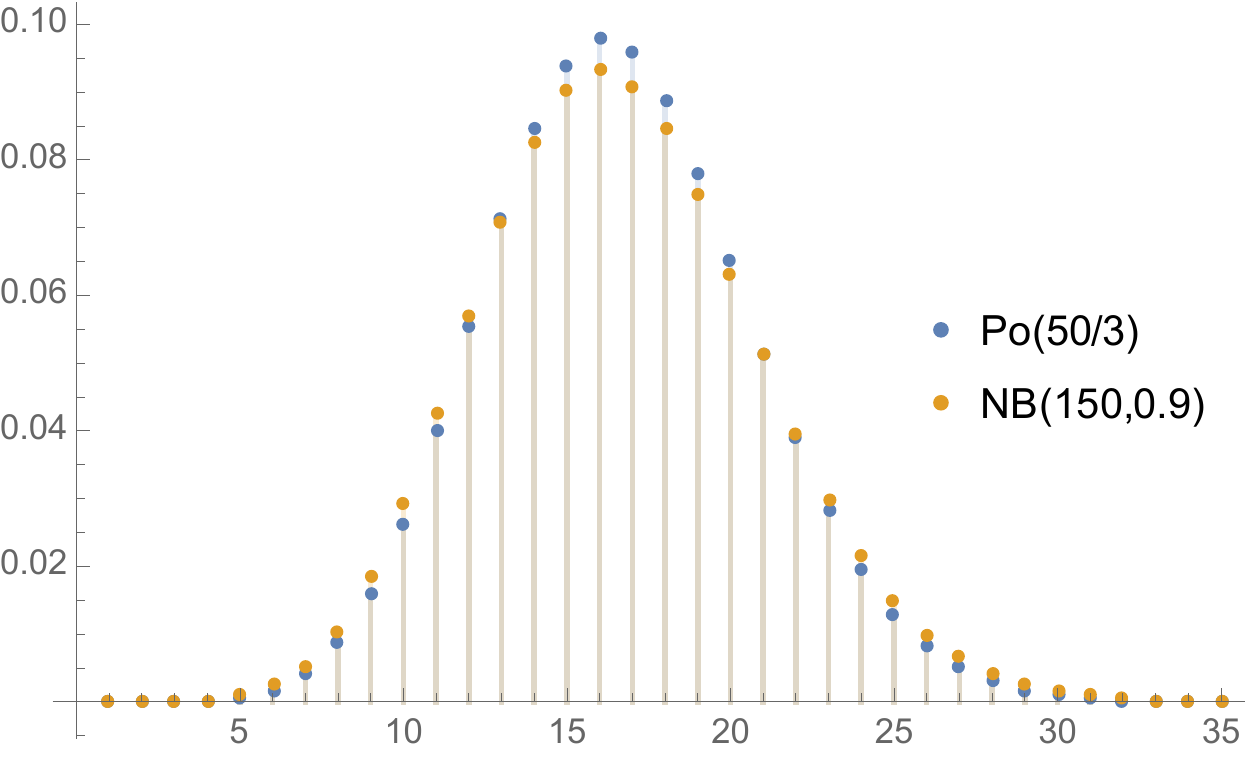}
  \captionof{figure}{$n=30, q=0.1$}
\end{minipage}
\end{figure}

\begin{figure}[H]
\centering
\begin{minipage}{.4\textwidth}
  \centering
  \includegraphics[width=0.9\linewidth]{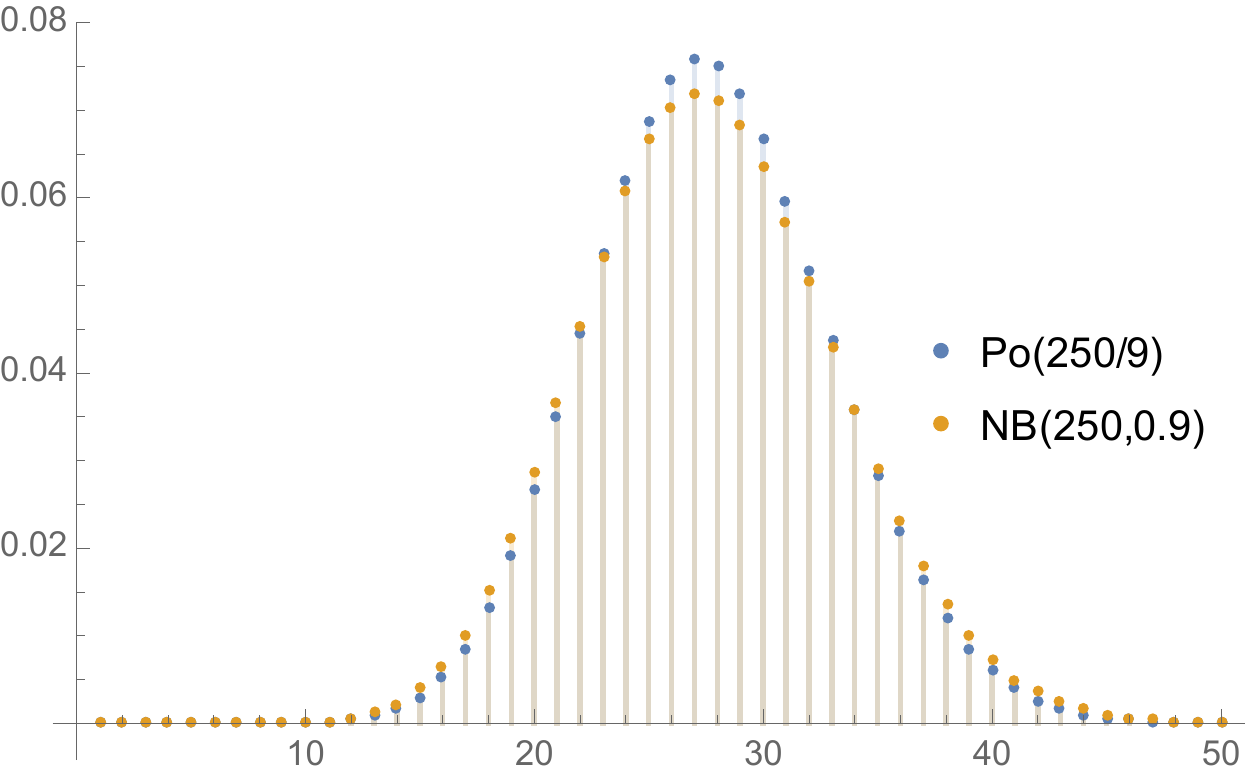}
  \captionof{figure}{$n=50, q=0.1$}
\end{minipage}%
\begin{minipage}{.4\textwidth}
  \centering
  \includegraphics[width=0.9\linewidth]{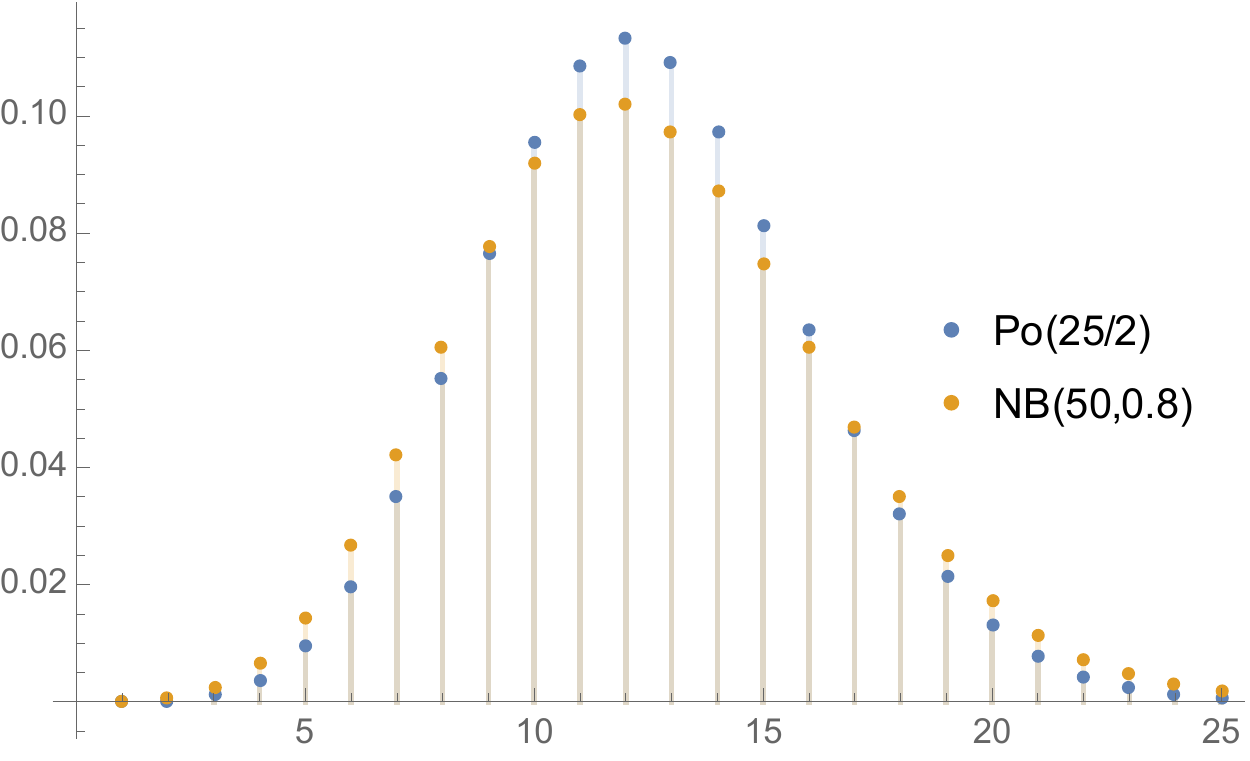}
  \captionof{figure}{$n=10, q=0.2$}
\end{minipage}
\end{figure}

\begin{figure}[H]
\centering
\begin{minipage}{.4\textwidth}
  \centering
  \includegraphics[width=0.9\linewidth]{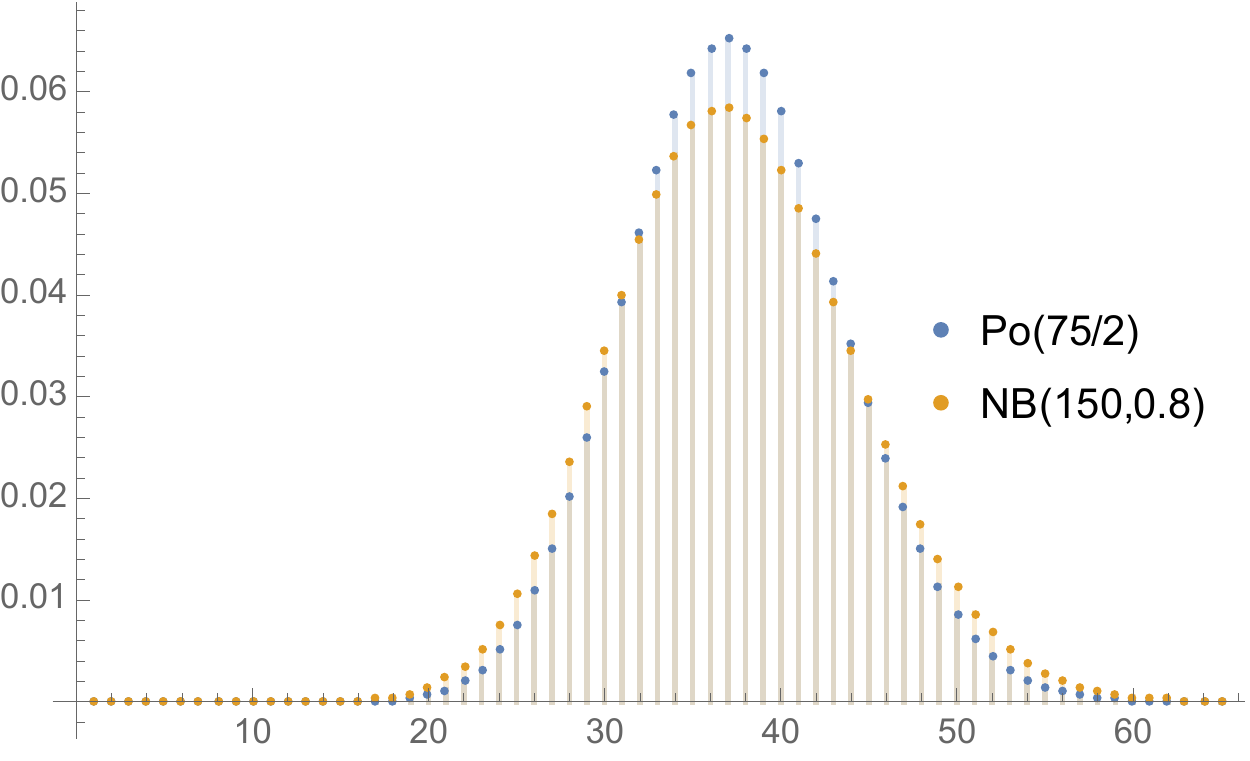}
  \captionof{figure}{$n=30, q=0.2$}
\end{minipage}%
\begin{minipage}{.4\textwidth}
  \centering
  \includegraphics[width=0.9\linewidth]{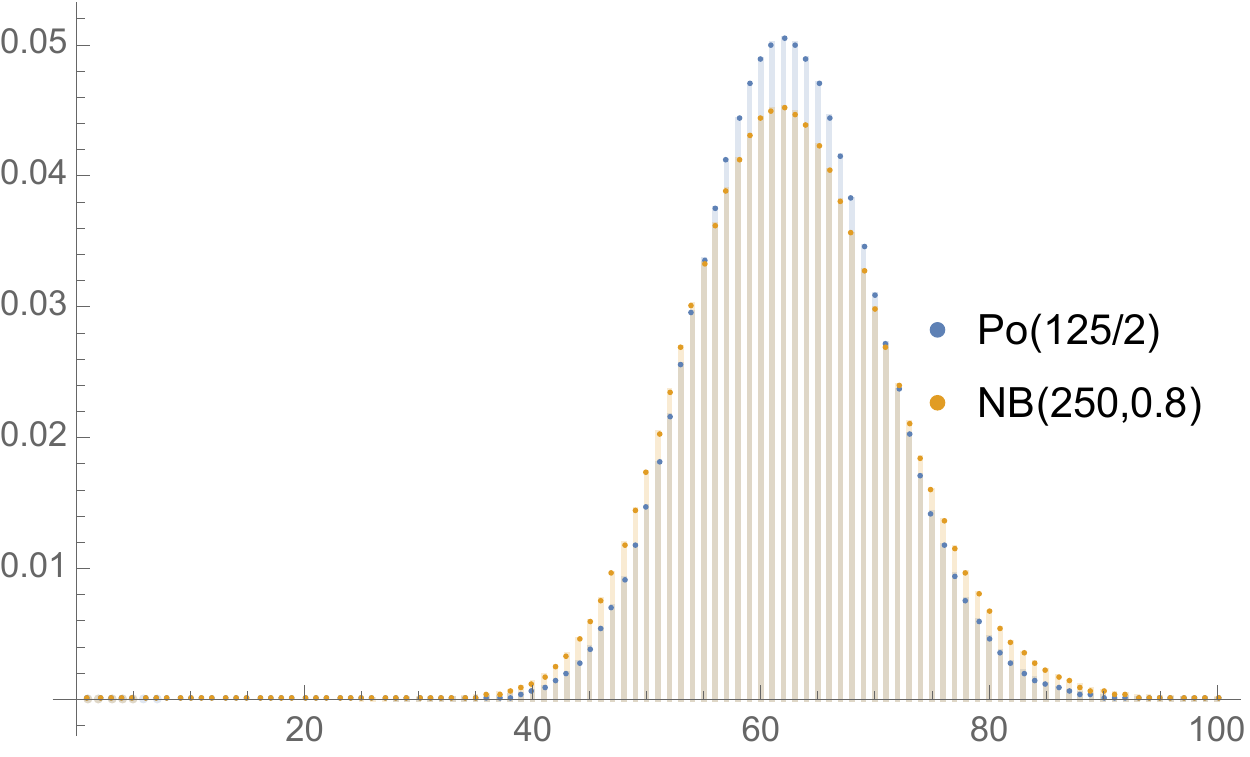}
  \captionof{figure}{$n=50, q=0.2$}
\end{minipage}
\end{figure}

\noindent
The above graphs are obtained by using the moment matching conditions. Also, from the numerical table and graphs, observe that the distributions are closer for sufficiently small values of $q$ and large values of $n$, as expected. 
\item From Theorem 1 of Hung and Giang (2016), it is given that, for $A \subset {\mathbb Z_+}$,
\begin{align}
\sup_{A}& \left|P(W_n \in A)- \sum_{k \in A} \frac{\lambda_n^k e^{-\lambda_n}}{k!} \right| \nonumber \\
&\le \sum_{i=1}^n \min \left\{ \lambda_n^{-1} (1-e^{-\lambda_n}) r_{n,i}(1-p_{n,i}), 1-p_{n,i} \right\} (1-p_{n,i})p_{n,i}^{-1}, \label{h}
\end{align}
where $W_n = \sum_{i=1}^n X_{n,i}, ~X_{n,i} \sim NB(r_{n,i},p_{n,i}) $ with $\lambda_n = {\mathbb E}(W_n)$. Note that if $\min \left\{ \lambda_n^{-1} (1-e^{-\lambda_n}) r_{n,i}(1-p_{n,i}), 1-p_{n,i}\right\}=1-p_{n,i}$, for all $i=1,2,\ldots,n$, then
\begin{align}
\sup_{A}& \left|P(W_n \in A)- \sum_{k \in A} \frac{\lambda_n^k e^{-\lambda_n}}{k!} \right| \le \sum_{i=1}^n (1-p_{n,i})^2 p_{n,i}^{-1}, \label{h1}
\end{align}
which is of order $O(n)$. Clearly, for large values of $n$, Theorem \ref{bound-po} is an improvement over \eqref{h1}.
\end{enumerate}
\end{enumerate}

\singlespacing

\end{document}